\documentclass[12pt]{article}
\usepackage{amsmath,amssymb,amsthm,mathtools,enumitem, hyperref}
\usepackage{authblk}
\usepackage{color}
\usepackage{geometry}
\newtheorem{Theorem}{Theorem}[section]

\newtheorem{Lemma}{Lemma}[section]

\begin{document}

\def\eR{\mathbf{R}}
\def\Rd{{\eR}^d}
\def\Rdd{{\eR}^{d\times d}}
\def\Rdsym{{\eR}^{d\times d}_{sym}}
\def\eN{\mathbf{N}}
\def\eZ{\mathbf{Z}}
\def\LND{L^2_{\dvr}}
\def\CrlE{\mathcal{E}}
\def\CrlH{\mathcal{H}}
\def\CrlG{\mathcal{G}}
\def\IdM{\operatorname{I}}
\def\IdOp{\operatorname{Id}}
\def\LTD{\mathcal L^2_{n,\dvr}}
\def\LTDD{\dot{\mathcal L}^2_{\dvr}}
\def\WN{W^{1,2}_{n}}
\def\WND{W^{1,2}_{0,\dvr}}
\def\WDD{\dot{W}^{1,2}}
\def\WNDT{W^{1,2}_{n,\dvr}}
\def\dist{\operatorname{dist}}
\def\dd{\mbox{d}}
\def\VS{\mathcal{V}}
\def\WS{\mathcal{W}}
\def\Meas{\mathcal{M}}
\def\NnMeas{\mathcal{M}^+}
\newcommand{\TR}{\operatorname{Tr}}
\newcommand{\essinf}{\operatorname{ess\,inf}}
\newcommand{\esssup}{\operatorname{ess\,sup}}
\newcommand{\supp}{\operatorname{supp}}
\newcommand{\spn}{\operatorname{span}}
\newcommand\dx{\; \mathrm{d}x}
\newcommand\dy{\; \mathrm{d}y}
\newcommand\dz{\; \mathrm{d}z}
\newcommand\dt{\; \mathrm{d}t}
\newcommand\ds{\; \mathrm{d}s}
\newcommand\diff{\mathrm{d}}
\newcommand\dvr{\mathop{\mathrm{div}}\nolimits}
\newcommand\diam{\mathrm{diam}}
\newcommand\tder{\partial_t}
\newcommand\tr{\mathop{\mathrm{tr}}\nolimits}
\newcommand\curlP{\mathop{\mathcal{P}}\nolimits}
\newcommand\curlQ{\mathop{\mathcal{Q}}\nolimits}
\newcommand\lspan{\mathop{\mathrm{span}}\nolimits}
\newcommand{\IED}{IED}
\newcommand\Hext{H_{ext}}
\newcommand\Adj{\mathrm{Adj}}
\newcommand\rot{\mathrm{rot}}

\title{On dissipative solutions to a system arising in viscoelasticity}
\author{Martin Kalousek}
\date{}
\affil{\small{Institute of Mathematics, University of W\"urzburg, Emil-Fischer-Str. 40, 97074 W\"urzburg, Germany\\ email: \textsl{martin.kalousek@mathematik.uni-wuerzburg.de}}}
\maketitle
\begin{abstract}
 We consider a model for an incompressible visoelastic fluid. It consists of the Navier-Stokes equations involving an elastic term in the stress tensor and a transport equation for the evolution of the deformation gradient. The novel feature of the paper is the introduction of the notion of a dissipative solution and its analysis. We show that dissipative solutions exist globally in time for arbitrary finite energy initial data and that a dissipative solution and a strong solution emanating from the same initial data coincide as long as the latter exists.
\end{abstract}
\section{Introduction}
We consider $T>0$ and a smooth bounded domain $\Omega\subset\eR^d$, $d=2,3$, and analyze the following system of partial differential equations 
\begin{equation}\label{ClassForm}
	\begin{aligned}
		\tder u+\dvr(u\otimes u)-\Delta u+\nabla p&=\dvr(FF^T)\\
		\dvr u&=0,\\
		\tder F+\dvr(u\otimes F)&=\nabla u F,\\
		\dvr F&=0
	\end{aligned}
\end{equation}
for the velocity of the fluid $u\!:(0,T)\times\Omega\to\eR^d$, the pressure $p\!:(0,T)\times\Omega\to \eR$ and the deformation gradient $F\!:(0,T)\times\Omega\to\eR^{d\times d}$. System \eqref{ClassForm} is supplemented by the initial data
\begin{equation*}
	u(0,\cdot)=u_0,\ F(0,\cdot)=F_0 \text{ in }\Omega
\end{equation*}
with $\dvr u_0=0$, $\dvr F_0=0$ in $\Omega$. We also add the homogeneous Dirichlet boundary conditions for the velocity, i.e.,
\begin{equation*}
	u=0\text{ on }(0,T)\times\partial\Omega.
\end{equation*}
Concerning the derivation of the above system we note that equation \eqref{ClassForm}$_1$ has been obtained via the energetic variational procedure in the special case of the Hookean elasticity, see e.g. \cite[Sec. 1.3]{LiLiZh2005}, while \eqref{ClassForm}$_3$ is just an expression of the time derivative of the deformation gradient, which is originally formulated in Lagrangian coordinates, in Eulerian coordinates, for details see e.g. \cite[Sec. 2]{LiWa01}.
Its physical background is discussed in \cite{Gu81} and \cite{La99}, whereas results of numerical simulations based on system \eqref{ClassForm} are presented in \cite{BaGoOb04,KuMa00,YuSasSe07}. 

The theory concerning strong solutions to \eqref{ClassForm} has been widely developed.
The local in time existence of strong solutions together with a blow-up criterion for a strong solution on the whole space, torus or a smooth bounded domain in $\eR^d$, $d=2,3$ was proven in \cite{LiLiZh2005}. Introducing an auxiliary vector field that replaces the deformation gradient, the authors of the latter paper showed also the global in time existence of strong solution for initial data satisfying certain smallness assumption on the whole space or torus in $\eR^2$. The same result was later obtained in \cite{LeLiZh2008} also for the three dimensional situation using a different technique that allows to work directly with the deformation gradient. Also the asymptotic behavior of a strong solution to system \eqref{ClassForm} was studied. The authors of \cite{LiZh08} showed the exponential decay for a strong solution with initial data sufficiently close to the equilibrium. The spatial regularity of strong solutions from results mentioned above is formulated in the framework of Sobolev spaces. The well-posedness of the Cauchy problem for \eqref{ClassForm} was shown in a functional setting invariant by the scaling of \eqref{ClassForm} in terms of the Besov spatial regularity in \cite{ZhFa12}. 

The goal of this paper is to investigate the existence of global in time solutions to \eqref{ClassForm} possibly with general initial data. The main difficulty is to show the convergence of the product $FF^\top$ for suitable approximations at least in the sense of distributions, a problem that has not been solved for weak solutions yet. In order to circumvent this difficulty we introduce a new class of dissipative solutions to \eqref{ClassForm} inspired by the same concept for a hyperbolic system of equations in \cite{FeRoScZa18}. We note that this concept of solution for incompressible Euler system was introduced in \cite{Li96}. Roughly explained, the dissipative solution satisfies \eqref{ClassForm} in the sense of integral identities with smooth test functions whose part corresponding to the right hand side of \eqref{ClassForm}$_1$ contains an extra term regarded as a defect measure. Moreover, a function called dissipation defect appears in the energy inequality. This dissipation defect is attributed to singularities that may hypothetically emerge during the fluid evolution. It dominates in a certain sense the additional term on the right hand side of the integral formulation of \eqref{ClassForm}$_1$, see \eqref{DefMesEst}. 

We end the introductory part with the outline of the paper. In Section \ref{Sec:Stat}, after necessary preliminaries we state a precise definition of a dissipative solution to \eqref{ClassForm} and formulate the main results of the paper concerning the global in time existence of a dissipative solution and the dissipative-strong uniqueness for solutions to \eqref{ClassForm}. 
Section \ref{Sec:DisExist} is devoted to the proof of the existence result and in Section \ref{Sec:DSUniq} the uniqueness result is proven. Finally, the Appendix contains several assertions about the space $\CrlH(\Omega)$.
\section{Formulation of the results}\label{Sec:Stat}
Let us introduce the notation used further. By $B(x,r)$ we mean the ball centered at $x$ with radius $r$. For vectors and matrices the scalar product is denoted by $\cdot$, a centered dot. For a vector $a\in\eR^l$ and a matrix $B\in\eR^{m\times n}$ the outer product $a\otimes B$ denotes the tensor with components $a_iB_{jk}$, $i=1,\ldots, l$, $j=1,\ldots, m$, $k=1,\ldots,n$. By $\mathcal{M}(\overline\Omega)$ we mean the space of Radon measures on $\overline\Omega$ and $\mathcal{M}^+(\overline\Omega)$ stands for the space of nonnegative Radon measures on $\overline\Omega$ for $\Omega\subset\eR^d$ open. For $k\in\eN$ and $q\in[1,\infty]$, $(L^q(\Omega),\|\cdot\|_q)$ and $(W^{k,q}(\Omega),\|\cdot\|_{k,q})$ denote the standard Lebesgue and Sobolev spaces. If $X$ is a Banach space of scalar functions then $X^m$ stands for a space of vector-valued functions with $m$ components each belonging to $X$. In a similar way, $X^{m\times n}$ is a space of matrix-valued functions. In order to keep the notation short, we write e.g. $ \|\cdot \|_{L^q(\Omega)}$ instead of $\|\cdot \|_{L^q(\Omega)^m}$, $ \|\cdot \|_{W^{k,q}(\Omega)}$ instead of $\|\cdot \|_{W^{k,q}(\Omega)^{m\times n}}$, and $\|\cdot \|_{L^q(0,t;L^r(\Omega))}$ instead of $\|\cdot \|_{L^q(0,t;L^r(\Omega)^{m})}$, etc. For Banach spaces $X,Y$ the notation $X\hookrightarrow Y$, $X\stackrel{C}{\hookrightarrow} Y$, is used for expressing embedding of $X$ to $Y$ that is continuous, compact respectively. The dual of $X$ is denoted $X^*$ and the notation $\left\langle\cdot,\cdot\right\rangle$ is used for the corresponding duality pairing. By $C_w([0,T];X)$ we mean the set of functions $f\in L^\infty(0,T;X)$ such that the real-valued mapping $t\mapsto\left\langle \phi, f(t)\right\rangle$ is continuous on $[0,T]$ for any $\phi\in X^*$. Further, we set 
\begin{equation*}
\begin{split}
		\LND(\Omega)&=\overline{\{v\in C^\infty_c(\Omega)^d;\ \dvr v=0\text{ in }\Omega\}}^{\|\cdot\|_{L^2}},\\
		\CrlH(\Omega)&=\{\Phi\in L^2(\Omega)^{d\times d};\ \dvr \Phi=0\text{ in }\Omega\}.
		\end{split}
\end{equation*}
Let us note that the distributional divergence is considered in the above expression, i.e., for a $d\times d$ matrix-valued $\Phi$
\begin{equation*}
	\left\langle \dvr\Phi,\varphi\right\rangle=-\int_\Omega \Phi\cdot(\nabla\varphi)^\top,\text{ for all }\varphi\in C^\infty_c(\Omega)^d.
\end{equation*} 
For a smooth $\Phi$ we define $(\dvr(\Phi))_j=\sum_{i=1}^d\partial_i\Phi_{ij}$.
The following notation is used for the subspaces of $W^{1,2}(\Omega)^d$, $W^{1,2}(\Omega)^{d\times d}$ respectively:
\begin{align*}
	\WS(\Omega)&=\{\Phi\in W^{1,2}(\Omega)^{d\times d};\dvr \Phi=0\text{ in }\Omega\},\\
	\WND(\Omega)^d&=\overline{\{v\in C^\infty_c(\Omega)^d: \dvr v =0\text{ in }\Omega\}}^{\|\cdot\|_{W^{1,2}}},\\
	\VS(\Omega)&=\WND(\Omega)^d\cap W^{3,2}(\Omega)^d.
\end{align*} 
We note that the embedding $\VS\hookrightarrow C^1(\overline{\Omega})^d$, $d=2,3$ holds.
By $\omega$ we denote a mollifier, i.e., $\omega\in C^\infty_c(B(0,1))$, $\omega\geq 0$, $\int_{\eR^d}\omega=1$ and define for $\rho>0$ $\omega_\rho(\cdot)=\rho^{-d}\omega\left(\frac{\cdot}{\rho}\right)$.
We continue with the introduction of the notion of a dissipative solution. Assuming that 
\begin{equation}\label{InitCondRegularity}
	u_0\in \LND(\Omega),\ F_0\in \CrlH(\Omega)
\end{equation}
we define a dissipative solution to \eqref{ClassForm} with initial conditions \eqref{InitCondRegularity} and dissipation defect $\mathcal{D}\geq 0$, $\mathcal{D}\in L^\infty(0,T)$ as a pair $(u,F)$ enjoying the regularity
\begin{equation*}
	u\in C_w(0,T;\LND(\Omega))\cap L^2(0,T;\WND(\Omega)^d),\ F\in C_w(0,T;\CrlH(\Omega))
\end{equation*} 
satisfying the energy inequality
\begin{equation}\label{DSEI}
	\frac{1}{2}(\|u(t)\|^2_{L^2(\Omega)}+\|F(t)\|^2_{L^2(\Omega)})+\int_0^t\|\nabla u(s)\|^2_{L^2(\Omega)}\ds+\mathcal{D}(t)\leq\frac{1}{2}(\|u_0\|^2_{L^2(\Omega)}+\|F_0\|^2_{L^2(\Omega)})
\end{equation}
for a.a. $t\in(0,T)$ and 
\begin{equation}\label{WeakForms}
\begin{split}
	(u(t),\psi(t))-(u_0,\psi(0))&=\int_0^t(u,\tder\psi)+(u\otimes u,\nabla\psi)-(\nabla u+FF^T,\nabla\psi)\ds +\int_0^t\left\langle\mathcal{R}^M,\nabla \psi\right\rangle\ds,\\
	(F(t),\Psi(t))-(F_0,\Psi(0))&=\int_0^t(F,\tder\Psi)+(u\otimes F,\nabla \Psi)+(\nabla uF,\Psi)\ds
\end{split}
\end{equation}
for all $t\in(0,T)$, $\Psi\in C^\infty([0,T]\times\overline{\Omega})^{d\times d}$, $\psi\in C^\infty_c([0,T]\times\Omega)^d$ with $\dvr\psi=0$, $\dvr\Psi=0$ in $(0,T)\times\Omega$. The corrector $\mathcal{R}^M\in L^\infty(0,T;\Meas(\overline\Omega)^{d\times d})$
satisfies
\begin{equation}\label{DefMesEst}
	\int_0^t\|\mathcal{R}^M(s)\|_{\Meas(\overline\Omega)}\ds\leq c\int_0^t\mathcal{D}(s)\ds.
\end{equation} 
The initial conditions are attained in the sense
\begin{equation}\label{InCondAtt}
	\lim_{t\to 0_+}\|u(t)-u_0\|^2_{L^2(\Omega)}+\|F(t)-F_0\|^2_{L^2(\Omega)}=0.
\end{equation}
For the sake of clarity the notation $(u,v)=\int_\Omega u(x)\cdot v(x)\dx$ is used.	Generic constants are denoted by $c$. For $t>0$ and $\Omega\subset\eR^d$ we use the notation $Q_t$ for the time-space cylinder $(0,t)\times \Omega$. Having all ingredients introduced we can state the main results of the paper.
\begin{Theorem}\label{Thm:Exis}
	For an arbitrary $T\in(0,\infty)$, a bounded Lipschitz domain $\Omega\subset\eR^d$, $d=2,3$, $u_0$ and $F_0$ satisfying \eqref{InitCondRegularity} there exists a dissipative solution to problem \eqref{ClassForm}.
\end{Theorem}
Dissipative solutions do not provide any information about the form of neither the dissipation defect $\mathcal{D}$ nor the corrector $\mathcal{R}^M$. Therefore such a notion of solution might seem very weak. Nevertheless, the dissipative solution to \eqref{ClassForm} can be related to a strong solution to \eqref{ClassForm} whose existence has been extensively investigated for various types of boundary conditions for the velocity, see e.g. \cite{LeLiZh2008, LiLiZh2005}. For the existence result for the problem with the homogeneous Dirichlet boundary condition we refer to \cite[Theorem 2.2]{LiLiZh2005}. The relation between the dissipative and strong solutions to \eqref{ClassForm} emanating from the same initial data is formulated in the following theorem.
\begin{Theorem}\label{Thm:DisStrongUniq}
	Let $\Omega\subset\eR^d$, $d=2,3$, be a smooth bounded domain and the initial data enjoy the regularity
	\begin{equation*}
		u_0\in \VS(\Omega),\quad F_0\in \WS(\Omega)\cap W^{3,2}(\Omega)^{d\times d}.
	\end{equation*}
	Let $(u,F)$ be a dissipative solution to \eqref{ClassForm} and $(\tilde u,\tilde F)$ be a strong solution to \eqref{ClassForm}, i.e., 	
	\begin{equation*}
		\begin{split}
			&\tilde u\in L^\infty(0,T;\VS(\Omega))\cap L^2(0,T;W^{4,2}(\Omega)^d),\quad \tder \tilde u\in L^2(0,T;W^{2,2}(\Omega)^d)  \\
			&\tilde F\in L^\infty(0,T; \WS(\Omega)\cap W^{3,2}(\Omega)^{d\times d}),\quad \tder \tilde F\in L^\infty(0,T; W^{1,2}(\Omega)^{d\times d})
			\end{split}
	\end{equation*}
	and \eqref{WeakForms} are satisfied with $\mathcal{R}^M=0$ in $(0,T)\times\Omega$. Then it follows that $(u,F)=(\tilde u,\tilde F)$ a.e. in $(0,T)\times\Omega$.
\end{Theorem}

\section{Proof of the existence theorem}\label{Sec:DisExist}
\subsection{Approximative system}
We begin with the introduction of an approximative system to \eqref{ClassForm}. Namely, we consider for $\varepsilon>0$ the system
\begin{equation}\label{ApproxSysClass}
	\begin{alignedat}{2}
		\tder u+\dvr(u\otimes u)-\Delta u+\nabla p&=\dvr(FF^\top),\\
		\dvr u&=0,\\
		\tder F+\dvr(u\otimes F)&=\nabla uF+\varepsilon (\Delta F-(\nabla\pi)^\top),\\
		\dvr F&=0
	\end{alignedat}
\end{equation}
in $(0,T)\times\Omega$ with boundary conditions $u=0$ and $(\nabla F-\pi\otimes\IdM)n=0$ on $(0,T)\times\partial\Omega$ and initial conditions $u(0,\cdot)=u_0$ and $F(0,\cdot)=F_0$ in $\Omega$ with $\dvr u_0=0$ and $\dvr F_0=0$ in $\Omega$. We note that the function $\pi$ appearing in \eqref{ApproxSysClass} is just a multiplier whose presence is due to the constraint $\dvr F=0$. 

The first task is to show the existence of a solution to \eqref{ApproxSysClass} which is done in the following lemma. 
\begin{Lemma}\label{Lem:ApproxEx}
	Under the assumptions of Theorem \ref{Thm:Exis} there exists a weak solution to approximate problem \eqref{ApproxSysClass}, i.e., a pair $(u, F)$ possessing the regularity 
	\begin{equation*}
	\begin{split}
		&u\in L^\infty (0,T;\LND(\Omega))\cap L^2(0,T;\WND(\Omega)^d),\ \tder u\in L^1(0,T;(\WND(\Omega)^d)^*)\\
		&F\in L^\infty(0,T;\CrlH(\Omega))\cap L^2(0,T;\WS(\Omega)),\ \tder F\in L^1(0,T;(\WS(\Omega))^*)
	\end{split}
	\end{equation*}
	and satisfying
	\begin{equation}\label{UFAppSys}
	\begin{alignedat}{2}
			\left\langle\tder u,\omega\right\rangle-\left(u\otimes u,\nabla \omega\right)+\left(\nabla u,\nabla\omega\right)+\left(FF^\top,\nabla\omega\right)=&0 &&\text{ for all }\omega\in \WND(\Omega)^d,\\
		\left\langle \tder F,\Phi\right\rangle-(u\otimes F,\nabla \Phi)-(\nabla u F,\Phi)+\varepsilon(\nabla F,\nabla\Phi)=&0 &&  \text{ for all }\Phi\in \WS(\Omega).
	\end{alignedat}
	\end{equation}
a.e. in  $(0,T)$.
	Moreover, $(u, F)$ fulfills
	\begin{equation}\label{EnergyBound}
	\begin{split}
		\frac{1}{2}\left(\|u(t)\|^2_{L^2(\Omega)}+\|F(t)\|^2_{L^2(\Omega)}\right)+\int_0^t\|\nabla u\|^2_{L^2(\Omega)}+\varepsilon\|\nabla F\|^2_{L^2(\Omega)}\leq\frac{1}{2}\left(\|u_0\|_{L^2(\Omega)}^2+\|F_0\|_{L^2(\Omega)}^2\right)
	\end{split}
\end{equation}
for a.a. $t\in(0,T)$ and
\begin{equation}\label{InCondAttPApp}
	\lim_{t\to 0_+}\|u(t)-u_0\|^2_{L^2(\Omega)}+\|F(t)-F_0\|^2_{L^2(\Omega)}=0.
\end{equation}
\end{Lemma}
\begin{proof}
The proof of the existence is performed in several steps. First Galerkin approximations for the velocity and the deformation gradient and an approximative system are introduced and their existence is shown. The second step consists in collecting estimates that are uniform with respect to the Galerkin index. In the third step we perform the limit passage and discuss the attainment of initial data.

Let us begin with the introduction of Galerkin approximations and an approximating system. We consider $\{\omega^i\}_{i=1}^\infty$, an orthonormal basis of $\LND(\Omega)$ being simultaneously an orthogonal basis of $\VS(\Omega)$. We note that elements of such a basis are constructed in the standard way as eigenfunctions of the following problem:
\begin{equation*}
	\sum_{k=1}^3 (\nabla^k\omega,\nabla^k w)=\lambda (\omega,w)\text{ for all }w\in\VS(\Omega),
\end{equation*} 
c.f. \cite[Theorem 4.11]{MaNeRoRu96}. Moreover, by $P^n$ we denote the projection of $\LND(\Omega)$ on $\spn\{\omega^1,\ldots,\omega^n\}$. Next we consider also $\{\Phi^j\}_{j=1}^\infty$ orthonormal basis of $\CrlH(\Omega)$ that is simultaneously orthonormal basis in $\WS(\Omega)$. The existence of such a basis is discussed in the Appendix, see Theorem \ref{Thm:Basis}. The projection of $\CrlH(\Omega)$ on $\spn\{\Phi^1,\ldots,\Phi^n\}$ is denoted by $\overline{P^n}$.\\
\textbf{Step 1:} We construct Galerkin approximations.
For fixed $n\in\eN$ we are looking for a pair $(u^n,F^n)$ defined as
\begin{align*}
	u^n(t,x)&=\sum_{k=1}^nc^{n}_k(t)\omega^k(x),\ F^n(t,x)=\sum_{k=1}^nd^n_k(t)\Phi^k(x),
\end{align*}
where the functions $c^n=(c^n_1,\ldots, c^n_n)$ and $d^n=(d^n_1,\ldots, d^{m,n}_n)$ satisfy in $(0,T)$ for each $i=1,\ldots, n$
\begin{equation}\label{GalSys}
	\begin{alignedat}{2}
		(\tder u^n,\omega^i)+\bigl((u^n\cdot\nabla)u^n,\omega^i\bigr)+(\nabla u^n,\nabla\omega^i)	-\bigl(F^n(F^n)^T,\nabla\omega^i\bigr)&=0,\\
		(\tder F^n,\Phi^i)+\bigl((u^n\cdot\nabla)F^n,\Phi^i\big)-\nabla u^nF^n,\Phi^i)+\varepsilon(\nabla F^n,\nabla\Phi^i)&=0,
	\end{alignedat}
\end{equation}
with initial conditions $u^n(0)=P^n(u_0)$, $F^n(0)=\overline{P^n}(F_0)$. We note that the existence of absolutely continuous functions $c^n$ and $d^n$ on $(0,t^*)\subset(0,T)$ follows by the Carath\'eodory existence theorem.\\
\textbf{Step 2:} We derive uniform estimates with respect to the Galerkin index $n$. First, we test \eqref{GalSys}$_1$ by $u^n$ to get 
\begin{equation}\label{BalancMomTest}
\frac{1}{2}\frac{\dd}{\dt}\|u^n\|_2^2+\|\nabla u^n\|_2^2=-\left(F^n(F^n)^\top,\nabla u^n\right).
\end{equation}
Testing \eqref{GalSys}$_2$ by $F^n$ yields
\begin{equation}\label{DefGradTest}
	\frac{1}{2}\frac{\dd}{\dt}\|F^n\|_2^2+\varepsilon\|\nabla F^n\|_2^2=\left(F^n(F^n)^\top,\nabla u^n\right).
\end{equation}
Summing up \eqref{BalancMomTest} and \eqref{DefGradTest} and integrating the resulting equality over $(0,t)\subset(0,t^*)$ we obtain 
\begin{equation}\label{EnergyEst}
\begin{split}
	&\frac{1}{2}\left(\|u^n(t)\|_{L^2(\Omega)}^2+\|F^n(t)\|_{L^2(\Omega)}^2\right)+\int_0^t\|\nabla u^n(s)\|_{L^2(\Omega)}^2+\varepsilon\|\nabla F^n(s)\|_{L^2(\Omega)}^2\ds\\
	&= \frac{1}{2}\left(\|u^n(0)\|_{L^2(\Omega)}^2+\|F^n(0)\|_{L^2(\Omega)}^2\right)\leq \frac{1}{2}\left(\|u_0\|_{L^2(\Omega)}^2+\|F_0\|_{L^2(\Omega)}^2\right),
	\end{split}
\end{equation}
which yields $\sup_{t\in(0,t^*)}\bigl(|c^n(t)|^2+|d^n(t)|^2\bigr)\leq c$ and $t^*=T$ accordingly. Moreover, we have
\begin{equation}\label{NUnifEst}
	\begin{split}
		\|u^n\|_{L^\infty(0,T;L^2(\Omega))}\leq& c,\\
		\|u^n\|_{L^2(0,T;W^{1,2}(\Omega))}\leq& c,\\
		\|F^n\|_{L^\infty(0,T;L^2(\Omega))}\leq& c,\\
		\|F^n\|_{L^2(0,T;W^{1,2}(\Omega))}\leq& c.
	\end{split}
\end{equation}

Next, we need bounds on time derivatives of $u^n$ and $F^n$. We start with the estimate of $\tder u^n$. To this end we pick $\omega\in L^{2}(0,T;\VS(\Omega))$ and estimate
\begin{equation*}
\begin{split}
\left|\int_0^{t^*}\left\langle \tder u^n,\omega\right\rangle\right|= \left|\int_0^{t^*}\left\langle \tder u^n,P^n\omega\right\rangle\right|\leq &\int_0^{t^*} |(u^n\otimes u^n,\nabla P^n\omega)|+|(\nabla u^n,\nabla P^n\omega)|\\&+|(F^n(F^n)^\top,\nabla P^n\omega)|\\=&\sum_{i=1}^{3}I^n_i.
\end{split}
\end{equation*}
Employing the H\"older inequality and the embedding $W^{3,2}(\Omega)$ to $W^{1,\infty}(\Omega)$ we get
\begin{align*}
I^n_1\leq& \int_0^T \|u^n\|^2_{L^2(\Omega)}\|\nabla P^n\omega\|_{L^\infty(\Omega)}\leq c\|u^n\|^2_{L^\infty(0,T;L^2(\Omega))}\|P^n\omega\|_{L^2(0,T;W^{3,2}(\Omega))}\\
\leq& c \|u^n\|^2_{L^\infty(0,T;L^2(\Omega))}\|\omega\|_{L^2(0,T;W^{3,2}(\Omega))},\\
I^n_2\leq&\|\nabla u^n\|_{L^2(Q_T)}\|\nabla P^n\omega\|_{L^2(Q_T)}\leq c\|\nabla u^n\|_{L^2(Q_T)}\|\omega\|_{L^2(0,T;W^{1,2}(\Omega))},\\
I^n_3\leq &\int_0^T\|F^n\|^2_{L^2(\Omega)}\|\nabla P^n\omega\|_{L^\infty(\Omega)}\leq c\|F^n\|^2_{L^\infty(0,T;L^2(\Omega))}\|\omega\|_{L^2(0,T;W^{3,2}(\Omega))}.
\end{align*}
Hence the estimate 
\begin{equation}\label{TDerUUnifBound}
	\|\tder u^n\|_{L^2(0,T;(\VS(\Omega))^*)}\leq c
\end{equation}
follows using \eqref{NUnifEst}.
It remains to show the bound on $\tder F^n$. Let us pick $\xi\in L^\frac{4}{4-d}(0,T;\WS(\Omega))$. From \eqref{GalSys}$_2$ we get
\begin{equation*}
	\left|\int_0^T \left\langle\tder F^n,\xi\right\rangle \right|\leq\int_0^T|(u^n\otimes F^n,\nabla \xi)|+|(\nabla u^nF^n,\xi)|+\varepsilon|(\nabla F^n,\nabla\xi)|=\sum_{i=1}^3I^n_3.
\end{equation*}
We estimate
\begin{align*}
I^n_1\leq &\int_0^{T} \|u^n\|_{L^4(\Omega)}\|F^n\|_{L^4(\Omega)}\|\nabla \xi\|_{L^2(\Omega)}\\
\leq &c\int_0^{T}\|u^n\|^{1-\frac{d}{4}}_{L^2(\Omega)}\|\nabla u^n\|^\frac{d}{4}_{L^2(\Omega)}\bigl(\|F^n\|^{1-\frac{d}{4}}_{L^2(\Omega)}\|\nabla F^n\|^\frac{d}{4}_{L^2(\Omega)}+\|F^n\|_{L^2(\Omega)}\bigr)\|\nabla \xi\|_{L^2(\Omega)}\\
\leq &c\|u^n\|^{1-\frac{d}{4}}_{L^\infty(0,T;L^2(\Omega))}\|F^n\|^{1-\frac{d}{4}}_{L^\infty(0,T;L^2(\Omega))}\bigl(\|\nabla u^n\|^\frac{d}{4}_{L^2(Q_{T})}+\|\nabla F^n\|^\frac{d}{4}_{L^2(Q_{T})}\bigr)\|\xi\|_{L^\frac{4}{4-d}(0,T;W^{1,2}(\Omega))}\\
&+c\|u^n\|^{1-\frac{d}{4}}_{L^\infty(0,T;L^2(\Omega))}\|F^n\|_{L^\infty(0,T;L^2(\Omega))}\|\nabla u^n\|^\frac{d}{4}_{L^2(Q_{T})}\|\xi\|_{L^\frac{8}{8-d}(0,T;W^{1,2}(\Omega))}\\
I^n_2\leq &\int_0^{T}\|\nabla u^n\|_{L^2(\Omega)}\|F^n\|_{L^{4}(\Omega)}\|\xi\|_{L^{4}(\Omega)}\\ \leq &c\int_0^{T}\|\nabla u^n\|_{L^2(\Omega)}(\|F^n\|^{1-\frac{d}{4}}_{L^2(\Omega)}\|\nabla F^n\|^\frac{d}{4}_{L^2(\Omega)}+\|F^n\|_{L^2(\Omega)})\|\xi\|_{W^{1,2}(\Omega)}\\
\leq& c\|F^n\|^{1-\frac{d}{4}}_{L^\infty(0,T;L^2(\Omega))}\|\nabla u^n\|_{L^2(Q_{T})}\|\nabla F^n\|^\frac{d}{4}_{L^2(Q_{T})}\|\xi\|_{L^\frac{8}{8-d}(0,T;W^{1,2}(\Omega))}\\&+c\|F^n\|_{L^\infty(0,T;L^2(\Omega))}\|\nabla u^n\|_{L^2(Q_{T})}\|\xi\|_{L^2(0,T;W^{1,2}(\Omega))}\\
I^n_3\leq &\varepsilon\|\nabla F^n\|_{L^2(Q_T)}\|\xi\|_{L^2(0,T;W^{1,2}(\Omega))}.
\end{align*}
Hence we get using bounds on $u^n$, $F^n$ from \eqref{NUnifEst} that
\begin{equation}\label{TDerFEst}
	\|\tder F^n\|_{L^\frac{4}{d}(0,T;\WS(\Omega))^*}\leq c.
\end{equation}
\textbf{Step 3:} Having estimates \eqref{NUnifEst}, \eqref{TDerUUnifBound} and \eqref{TDerFEst} at hand we may invoke the standard compactness arguments including reflexivity, the Banach-Alaoglu theorem and the Aubin-Lions Lemma with $\WND(\Omega)^d\stackrel{C}{\hookrightarrow} L^q(\Omega)^d\hookrightarrow (\WND(\Omega)^d)^*$, $\WS(\Omega)\stackrel{C}{\hookrightarrow}L^q(\Omega)^{d\times d}\hookrightarrow (\WS(\Omega))^*$ respectively, for $q<2^*$, the Sobolev exponent, to infer the existence of a not explicitly labeled subsequence $\{(u^n,F^n)\}_{n=1}^\infty$ such that as $n\to\infty$
\begin{equation}\label{NConv}
	\begin{alignedat}{2}
		u^n& \rightharpoonup^* u&&\text{ in }L^\infty(0,T; L^2(\Omega)^d),\\
		u^n& \rightharpoonup u&&\text{ in }L^2(0,T; \WND(\Omega)^d),\\
		\tder u^n&\rightharpoonup \tder u&&\text{ in }L^2(0,T;(\VS(\Omega))^*),\\
		u^n& \to u&&\text{ in }L^2(0,T;L^2(\Omega)^d),\\
		F^n& \rightharpoonup^* F&&\text{ in }L^\infty(0,T;L^2(\Omega)^{d\times d}),\\
		F^n& \rightharpoonup F&&\text{ in }L^2(0,T;\WS(\Omega)),\\
		\tder F^n&\rightharpoonup \tder F&&\text{ in }L^\frac{4}{d}(0,T;(\WS(\Omega))^*),\\
		F^n& \to F&&\text{ in }L^2(0,T;L^2(\Omega)^{d\times d}).
	\end{alignedat}
\end{equation}
Energy inequality \eqref{EnergyBound} is a consequence of \eqref{EnergyEst} and the convergences above.

Using the regularity of $\tder u$ we infer $u\in C([0,T];(W^{1,2}_{0,\dvr}(\Omega)^d)^*)$, c.f. \cite[Lemma~7.1]{Ro13}, and as $L^2_{\dvr}(\Omega)\hookrightarrow (W^{1,2}_{0,\dvr}(\Omega)^d)^*$ we get, see e.g. \cite[Ch.~III, Lemma~1.4]{Tem77}
\begin{equation}\label{ContTimeVel}
		u\in C_w([0,T];L^2_{\dvr}(\Omega)).
\end{equation}
Using the same arguments we deduce 
\begin{equation}\label{ContTimeDefG}
		F\in C_w([0,T];\CrlH(\Omega)).
\end{equation}
We focus on the attainment of the initial data. First, we show
\begin{equation}\label{InCond}
u(0)=u_0,\ F(0)=F_0\text{ a.e. in }\Omega.
\end{equation}
To obtain the first equality we integrate \eqref{GalSys}$_1$ with fixed $i\leq n$ over $(0,t)$ and pass to the limit $n\to\infty$ with the help of \eqref{NConv}$_{2,4,8}$ and $u^n(0)\to u_0$ in $L^2(\Omega)^d$ arriving at
\begin{equation*}
	(u(t)-u_0,\omega^i)=-\int_0^t\bigl((u\cdot\nabla)u,\omega^i\bigr)+(\nabla u,\nabla\omega^i)+(FF^\top,\nabla\omega^i).
\end{equation*}
As $\omega^i$ is a basis element of $\VS(\Omega)$, it can be replaced by an arbitrary $\omega\in\VS(\Omega)$. Using \eqref{ContTimeVel} we conclude $u(0)=u_0$ a.e. in $\Omega$. We proceed similarly to obtain the second statement in \eqref{InCond}.
From \eqref{EnergyBound} we get
\begin{equation*}
\frac{1}{2}\left(\|u(t)\|^2_{L^2(\Omega)}+\|F(t)\|^2_{L^2(\Omega)}\right)+\int_0^t\|\nabla u\|^2_{L^2(\Omega)}\leq\frac{1}{2}\left(\|u_0\|_{L^2(\Omega)}^2+\|F_0\|_{L^2(\Omega)}^2\right)
\end{equation*}
for all $t\in(0,T)$ due to \eqref{ContTimeVel} and \eqref{ContTimeDefG}. The latter inequality, the weak lower semicontinuity of the norm, \eqref{ContTimeVel}, \eqref{ContTimeDefG} and \eqref{InCond} imply
\begin{align*}
\|u_0\|^2_{L^2(\Omega)}+\|F_0\|^2_{L^2(\Omega)}&\leq \liminf_{t\to 0_+}\left(\|u(t)\|^2_{L^2(\Omega)}+\|F(t)\|^2_{L^2(\Omega)}\right)\leq \limsup_{t\to 0_+}\left(\|u(t)\|^2_{L^2(\Omega)}+\|F(t)\|^2_{L^2(\Omega)}\right)\\
&\leq\|u_0\|_{L^2(\Omega)}^2+\|F_0\|_{L^2(\Omega)}^2.
\end{align*}
We conclude \eqref{InCondAttPApp} by combining the latter chain of inequalities with \eqref{ContTimeVel} and \eqref{ContTimeDefG}.
\end{proof}

Having shown the existence of solution to approximative system \eqref{ApproxSysClass} and its estimates that are independent of the regularizing parameter $\varepsilon$ we collect convergences that follow immediately.

\begin{Lemma}\label{Lem:Convergences}
	Let assumptions of Theorem \ref{Thm:Exis} be satisfied and $\{\varepsilon^r\}_{r=1}^\infty$ be a sequence such that $\varepsilon^r\to 0_+$ as $r\to\infty$. Let $\{(u^r, F^r)\}_{r=1}^\infty$ be a sequence of solutions to \eqref{ApproxSysClass} with $\varepsilon=\varepsilon^r$ constructed in Lemma \ref{Lem:ApproxEx}. Then the following uniform estimates hold 
\begin{equation}\label{RUnifEst}
	\begin{split}
		\|u^r\|_{L^\infty(0,T;L^2(\Omega))}&\leq c,\\
		\|u^r\|_{L^2(0,T;W^{1,2}(\Omega))}&\leq c,\\
		\|F^r\|_{L^\infty(0,T;L^2(\Omega))}&\leq c
	\end{split}
\end{equation}
and there exist a not explicitly labeled subsequence of $\{(u^r, F^r)\}_{r=1}^\infty$, $u\in L^\infty(0,T;\LND(\Omega))\cap L^2(0,T;\WND(\Omega)^d)$, $F\in L^\infty(0,T;\CrlH(\Omega))$, $\overline{|F|^2}\in L^\infty(0,T;\NnMeas(\overline{\Omega}))$ and $\bar\sigma\in \NnMeas([0,T]\times\overline{\Omega})$ such that
	\begin{equation}\label{NConvergences}
		\begin{alignedat}{2}
			u^r&\rightharpoonup u&&\text{ in }L^2(0,T;\WND(\Omega)^d),\\
			u^r&\rightharpoonup^* u&&\text{ in }L^\infty(0,T;L^2(\Omega)^d),\\
			u^r&\to u&&\text{ in }L^2(0,T;L^2(\Omega)^d),\\
			\tder u^r&\rightharpoonup \tder u&&\text{ in }L^2(0,T;(\VS(\Omega))^*),\\
			F^r&\rightharpoonup^* F&&\text{ in }L^\infty(0,T;L^2(\Omega)^{d\times d}),\\
			\tder F^r&\rightharpoonup \tder F&&\text{ in }L^2(0,T;(\WS(\Omega)\cap W^{3,2}(\Omega)^d)^*),\\
			|F^r|^2&\rightharpoonup^* \overline{|F|^2}&&\text{ in }L^\infty(0,T;\NnMeas(\overline{\Omega})),\\
			|\nabla u^r|^2&\rightharpoonup^* \bar{\sigma}&&\text{ in }\NnMeas([0,T]\times\overline{\Omega})).
		\end{alignedat}
	\end{equation}
	\begin{proof}
		Let us consider the sequence of solutions $\{(u^r,F^r)\}_{r=1}^\infty$ to \eqref{ApproxSysClass} from the assertion of the lemma. Then the estimates in \eqref{RUnifEst} follow directly from \eqref{EnergyBound}. Moreover, from \eqref{UFAppSys}$_1$ we get for arbitrary $\phi\in L^2(0,T;\VS(\Omega))$
		\begin{equation*}
		\begin{split}
			\left|\int_0^T\left\langle \tder u^r,\phi\right\rangle\right|&=\left|\int_0^T(u^r\otimes u^r-\nabla u^r-F^r(F^r)^\top,\nabla\phi)\right|\\
			&\leq \int_0^T (\|u^r\|^2_{L^2(\Omega)}+\|F^r\|^2_{L^2(\Omega)})\|\nabla\phi\|_{L^\infty(\Omega)}+\|\nabla u^r\|_{L^2(\Omega)}\|\nabla\phi\|_{L^2(\Omega)}\\
			&\leq c\int_0^T (\|u^r\|^2_{L^2(\Omega)}+\|F^r\|^2_{L^2(\Omega)}+\|\nabla u^r\|_{L^2(\Omega)})\|\phi\|_{W^{3,2}(\Omega)}\\
			&\leq c(\|u^r\|^2_{L^4(0,T;L^2(\Omega))}+\|F^r\|^2_{L^4(0,T;L^2(\Omega))}+\|u^r\|_{L^2(0,T;W^{1,2}(\Omega))})\|\phi\|_{L^2(0,T;W^{3,2}(\Omega))}.
			\end{split}
		\end{equation*}
		Taking into account \eqref{RUnifEst} we deduce 
		\begin{equation}\label{TDerUE}
		\|\tder u^r\|_{L^2(0,T;(\VS(\Omega))^*)}\leq c.
		\end{equation}
		Similarly, fixing an arbitrary $\Phi\in L^2(0,T;\VS(\Omega))$ we get
		\begin{equation*}
		\begin{split}
			\left|\int_0^T\left\langle \tder F^r,\Phi\right\rangle\right|=&\left|\int_0^T(u^r\otimes F^r,\nabla\Phi)+(\nabla u^rF^r,\Phi)-\varepsilon^r( \nabla F^r,\nabla\Phi)\right|\\
			\leq &\int_0^T \|u^r\|_{L^2(\Omega)}\|F^r\|_{L^2(\Omega)}\|\nabla\phi\|_{L^\infty(\Omega)}+\|\nabla u^r\|_{L^2(\Omega)}\|F^r\|_{L^2(\Omega)}\|\Phi\|_{L^\infty(\Omega)}\\&+\varepsilon^r\|\nabla F^r\|_{L^2(\Omega)}\|\nabla \Phi\|_{L^2(\Omega)}\\
			\leq &c\int_0^T(\|u^r\|_{W^{1,2}(\Omega)}\|F^r\|_{L^2(\Omega)}+\varepsilon^r\|\nabla F^r\|_{L^2(\Omega)})\|\Phi\|_{W^{3,2}(\Omega)}\\
			\leq &c(\|u^r\|_{L^2(0,T;W^{1,2}(\Omega))}\|F^r\|_{L^\infty(0,T;L^2(\Omega))}+\varepsilon^r\|\nabla F^r\|_{L^2(\Omega)})\|\Phi\|_{L^2(0,T;W^{3,2}(\Omega))}
			\end{split}
		\end{equation*}
		implying
		\begin{equation*}
		\|\tder F^r\|_{L^2(0,T;(\WS(\Omega)\cap W^{3,2}(\Omega))^*)}\leq c.
		\end{equation*}		
The convergences in \eqref{NConvergences} are obtained as a direct consequence of \eqref{RUnifEst} and \eqref{TDerUE}. Namely, \eqref{NConvergences}$_3$ follows by the Aubin-Lions Lemma as $\WND(\Omega)^d\stackrel{C}{\hookrightarrow}L^2(\Omega)^d\hookrightarrow (\VS(\Omega))^*$. 
	\end{proof}
\end{Lemma}

\begin{proof}[Proof of Theorem \ref{Thm:Exis}]
We consider a sequence $\{\varepsilon^r\}_{r=1}^\infty$ such that $\varepsilon^r\to 0_+$ as $r\to \infty$. Applying Lemma \ref{Lem:Convergences} we find a sequence $\{(u^r,F^r)\}_{r=1}^\infty$ of solutions to \eqref{ApproxSysClass} with $\varepsilon=\varepsilon^r$ and a limit $(u,F)$. From the energy inequality \eqref{EnergyBound} we infer for a fixed $\tau\in(0,T)$ and $r\in\eN$
\begin{equation*}
\begin{split}
	&\frac{1}{2}\int_\Omega\left(|u(\tau)|^2+|F(\tau)|^2\right)+\int_0^\tau\int_\Omega|\nabla u|^2\\&+\frac{1}{2}\int_\Omega\left(|u^r(\tau)|^2-|u(\tau)|^2+|F^r(\tau)|^2-|F(\tau)|^2\right)+\int_0^\tau\int_\Omega\left(|\nabla u^r|^2-|\nabla u|^2\right)\\&\leq \frac{1}{2}\int_\Omega\left(|u_0|^2+|F_0|^2\right).
	\end{split}
\end{equation*}
Multiplying the latter inequality with $\theta\in C^\infty_c((0,T))$, $\theta\geq 0$ and integrating over $(0,T)$ we get 
\begin{equation*}
\begin{split}
	&\int_0^T\theta(\tau)\left(\frac{1}{2}\int_\Omega\left(|u(\tau)|^2+|F(\tau)|^2\right)+\int_0^\tau\int_\Omega|\nabla u|^2\right)\dd\tau\\&+\int_0^T\theta(\tau)\left(\frac{1}{2}\int_\Omega\left(|u^r(\tau)|^2-|u(\tau)|^2+|F^r(\tau)|^2-|F(\tau)|^2\right)+\int_0^\tau\int_\Omega\left(|\nabla u^r|^2-|\nabla u|^2\right)\right)\dd\tau\\&\leq \frac{1}{2}\int_0^T\theta(\tau)\dd\tau\int_\Omega\left(|u_0|^2+|F_0|^2\right).
	\end{split}
\end{equation*}
We pass to the limit $r\to\infty$ using \eqref{NConvergences} and arrive at
\begin{equation}\label{EILimWeak}
	\int_0^T\theta(\tau)\left(\frac{1}{2}\int_\Omega\bigl(|u(\tau)|^2+|F(\tau)|^2\bigr)+\int_0^\tau\int_\Omega|\nabla u|^2+\mathcal{D}(\tau)\right)\dd\tau\leq \frac{1}{2}\int_0^T\theta(\tau)\dd\tau\int_\Omega(|u_0|^2+|F_0|^2)
\end{equation}
with the dissipation defect $\mathcal{D}$ defined as
\begin{equation*}
	\mathcal{D}(t)= \mathcal{G}(t)(\overline\Omega)+\sigma([0,t]\times\overline\Omega)
\end{equation*}
for nonnegative measures 
\begin{equation*}
	\mathcal{G}(t) = \overline{|F|^2}(t)-|F(t)|^2\dx,\quad
	\sigma=\bar\sigma-|\nabla u|^2\dx\dt.
\end{equation*}
The fact that $\mathcal{D}\in L^\infty(0,T)$ follows immediately from the regularity of the limit objects $\overline{|F|^2}$, $\bar{\sigma}$, $u$ and $F$. The nonnegativity of $\mathcal{D}$ is a direct consequence of the weak lower semicontinuity of convex functionals. Fixing $t\in(0,T)$ and setting in \eqref{EILimWeak} $\theta(\tau)=\omega_\rho(t-\tau)$, where $\omega_\rho$ is a one-dimensional mollifier with $\rho<\frac{1}{2}\min\{t,T-t\}$, and letting $\rho\to 0_+$ we infer
\begin{equation}\label{EILim}
	\frac{1}{2}\left(\|u(t)\|_{L^2(\Omega)}^2+\|F(t)\|_{L^2(\Omega)}^2\right)+\int_0^t\|\nabla u\|_{L^2(\Omega)}^2\ds+\mathcal{D}(t)\leq \frac{1}{2}\left(\|u_0\|_{L^2(\Omega)}^2+\|F_0\|_{L^2(\Omega)}^2\right)
\end{equation}
for a.a. $t\in (0,T)$. Next we study the convergence of the sequence $\{F^r(F^r)^\top\}_{r=1}^\infty$ for which only an $L^1$ uniform estimate with respect to the space variable is available. From \eqref{RUnifEst}$_3$ we infer the existence of a not explicitly labeled subsequence $\{F^r(F^r)^\top\}_{r=1}^\infty$
and $\mathcal{R}^M\in L^\infty(0,T;\Meas(\overline\Omega)^{d\times d})$ such that
\begin{equation}\label{NonlinConv}
		F^r(F^r)^\top\rightharpoonup^* FF^\top+\mathcal{R}^M\text{ in }L^\infty(0,T;\Meas(\overline\Omega)^{d\times d})\text{ as }r\to\infty.
\end{equation}
With the help of the dissipation defect $\mathcal{D}$ we estimate the corrector $\mathcal{R}^M$.
Fixing an arbitrary $\Phi\in C(\overline{\Omega})^{d\times d}$ we can rewrite and estimate for $t\in(0,T)$
\begin{equation*}
	\begin{split}
	\int_0^t \left\langle F^r(F^r)^\top-FF^\top,\Phi\right\rangle\ds&=\int_0^t \left\langle (F^r-F)(F^r-F)^\top,\Phi\right\rangle+\int_\Omega \bigl(F(F^r-F)^\top+(F^r-F)F^\top\bigr)\cdot\Phi\ds\\&\leq \int_0^t \left\langle|F^r-F|^2,|\Phi|\right\rangle+\int_\Omega (F^r-F)\cdot\Phi^\top F+(F^r-F)\cdot\Phi F\ds.
	\end{split}
\end{equation*}
Performing the passage $r\to \infty$ employing \eqref{NConvergences}$_5$ and \eqref{NonlinConv} we obtain
\begin{equation*}
	\begin{split}
	\int_0^t \left\langle \mathcal{R}^M,\Phi\right\rangle\ds\leq \int_0^t \left\langle \overline{|F|^2}-|F|^2,|\Phi|\right\rangle\ds\leq \int_0^t \left\langle \mathcal{G},|\Phi|\right\rangle\ds
	\end{split}
\end{equation*}
and taking the supremum over $\Phi\in C(\overline{\Omega})^{d\times d}$ with $\|\Phi\|_{C(\overline\Omega)}\leq 1$ it follows that 
\begin{equation*}
		\int_0^t \|\mathcal{R}^M\|_{\Meas(\overline{\Omega})}\ds\leq \int_0^t\mathcal{D}(s)\ds\text{ for a.a. }t\in(0,T).
\end{equation*}
The next task is to show that the integral formulations in \eqref{WeakForms} are satisfied. To this end we fix an arbitrary $\psi\in C^\infty_c([0,T]\times\Omega)^d$ and $\Psi\in C^\infty([0,T]\times\overline{\Omega})^{d\times d}$ with $\dvr\psi=0$, $\dvr\Psi=0$ in $ Q_T$. Then we set $\omega=\psi(t)$, $\Phi=\Psi(t)$ in \eqref{UFAppSys} with $u=u^r$ and $F=F^r$, integrate by parts in time in the first terms and pass to the limit $r\to\infty$ using \eqref{NConvergences}$_{1,2,3,5}$ and \eqref{NonlinConv} to conclude \eqref{WeakForms}. We explain in detail the passage
\begin{equation}\label{PassR}
	\int_0^t\int_\Omega\nabla u^rF^r\cdot\Psi\to\int_0^t\int_\Omega\nabla uF\cdot\Psi\text{ as }r\to\infty.
\end{equation}
First, we observe that for $w\in C^\infty_c(\Omega)^d$, $G\in\CrlH(\Omega)$ and $\Xi\in C^\infty(\overline\Omega)^{d\times d}$ we have
\begin{equation}\label{AuxE}
\int_\Omega \nabla wG\cdot\Xi=-\int_\Omega (w\otimes G^\top)\cdot\nabla\Xi.
\end{equation}
Indeed, as the distributional divergence of $G$ vanishes in $\Omega$ we have using the Einstein summation convention 
\begin{equation*}
	\int_\Omega\partial_k w_iG_{kj}\Xi_{ij}=\int_\Omega G_{kj}(\partial_k(w_i\Xi_{ij})-w_i\partial_k\Xi_{ij})=-\int_\Omega G_{kj}w_i\partial_k\Xi_{ij}.
\end{equation*}
Obviously, for $u^r\in L^2(0,T;\WND(\Omega)^d)$ we can find $\{u^{r,\delta}\}_{\delta>0}\in L^2(0,T;C^\infty_c(\Omega)^d)$ such that
\begin{equation*}
	u^{r,\delta}\to u^r\text{ in }L^2(0,T;W^{1,2}_0(\Omega)^d) \text{ as }\delta\to 0_+.
\end{equation*}
Using the latter convergence and \eqref{AuxE} we get
\begin{equation*}
	\int_0^t\int_\Omega \nabla u^rF^r\cdot\Psi=\lim_{\delta\to 0_+}\int_0^t\int_\Omega \nabla u^{r,\delta}F^r\cdot\Psi=-\lim_{\delta\to 0_+}\int_0^t\int_\Omega u^{r,\delta}\otimes(F^r)^\top\cdot\nabla \Psi =-\int_0^t\int_\Omega u^{r}\otimes(F^r)^\top\cdot\nabla \Psi .
\end{equation*}
Hence \eqref{PassR} follows by using \eqref{NConvergences}$_{2,5}$ and the equality \eqref{AuxE} with $G=F(t)$, $w=u^\delta(t)$ and $\Xi=\Psi(t)$, where $\{u^\delta\}\subset L^2(0,T;C^\infty_c(\Omega)^d)$ approximates $u$ in $L^2(0,T;W^{1,2}_0(\Omega)^d)$.

Finally, we deal with the attainment of the initial data. The regularity of $\tder u$ provided by Lemma \ref{Lem:Convergences} implies $u\in C([0,T]; (\mathcal{V}(\Omega))^*)$, c.f. \cite[Lemma 7.1]{Ro13}. As $L^2_{\dvr}(\Omega)\hookrightarrow (\mathcal{V}(\Omega))^*$ we conclude 
\begin{equation}\label{TimeContU}
u\in C_w([0,T];L^2_{\dvr}(\Omega)),
\end{equation}
see \cite[Ch.~III, Lemma~1.4]{Tem77} for details. Obviously, performing the limit $t\to 0_+$ on both sides of \eqref{WeakForms}$_1$, where we set $\psi=\psi_1\psi_2$ for an arbitrary but fixed $\psi_1\in C^\infty([0,T])$ with $\psi_1(0)=1$ and $\psi_2\in C^\infty_c(\Omega)^d$ with $\dvr\psi_2=0$ in $Q_T$, we get 
\begin{equation*}
	(u(0)-u_0,\psi_2)=0\text{ for all } \psi_2\in C^\infty_c(\Omega)^d,\ \dvr\psi_2=0\text{ in }Q_T,
\end{equation*}
i.e., $u(0)=u_0$ a.e. in $\Omega$. We note that using similar arguments as above we infer
\begin{equation}\label{FContTime}
	F\in C_w([0,T];\CrlH(\Omega))
\end{equation} 
and 
\begin{equation}\label{FInVal}
	F(0)=F_0\text{ a.e. in }\Omega.
\end{equation}
From inequality \eqref{EILim} we get by \eqref{TimeContU} and \eqref{FContTime} for all $t\in(0,T)$
\begin{equation*}
	\frac{1}{2}\left(\|u(t)\|_{L^2(\Omega)}^2+\|F(t)\|_{L^2(\Omega)}^2\right)+\int_0^t\|\nabla u\|_{L^2(\Omega)}^2\ds\leq \frac{1}{2}\left(\|u_0\|_{L^2(\Omega)}^2+\|F_0\|_{L^2(\Omega)}^2\right).
\end{equation*}
Using \eqref{TimeContU}, \eqref{FContTime} and the weak lower semicontinuity of norms we conclude \eqref{InCondAtt} by repeating the arguments from the Step 3 of the proof of Lemma \ref{Lem:ApproxEx}.
\end{proof}

\section{Proof of Theorem \ref{Thm:DisStrongUniq}}\label{Sec:DSUniq}
For the sake of clarity we introduce the energy functional 
\begin{equation*}
	E(u,F)(t)=\frac{1}{2}\int_\Omega(|u(t)|^2+|F(t)|^2)+\int_0^t\int_\Omega|\nabla u|^2
\end{equation*}
and the relative energy functional that is understood as a specific distance of a solution $(u,F)$ and a generic pair $(\tilde u,\tilde F)$

\begin{equation*}
	\mathcal{E}\left(u,F|\tilde u,\tilde F\right)(t)=\frac{1}{2}\int_\Omega(|u(t)-\tilde u(t)|^2+|F(t)-\tilde F(t)|^2)+\int_0^t\int_\Omega|\nabla u-\nabla \tilde u|^2.
\end{equation*}

We expand
\begin{equation*}
	\mathcal{E}\left(u,F|\tilde u,\tilde F\right)(t)= E(u,F)(t) +E(\tilde u,\tilde F)(t)-(u(t),\tilde u(t))-(F(t),\tilde F(t))-2\int_0^t(\nabla u,\nabla\tilde u).
\end{equation*}
In order to conclude the dissipative-strong uniqueness for solutions to \eqref{ClassForm} we investigate the difference 
\begin{equation*}
	\left[\mathcal{E}\left(u,F|\tilde u,\tilde F\right)\right]_{\tau=0}^{\tau=t}=\mathcal{E}\left(u,F|\tilde u,\tilde F\right)(t)-\mathcal{E}\left(u,F|\tilde u,\tilde F\right)(0)
\end{equation*}
for a dissipative solution $(u,F)$ to \eqref{ClassForm} and the strong solution $(\tilde u,\tilde F)$ to \eqref{ClassForm}. 
Using the regularity of the strong solution $(\tilde u,\tilde F)$ to \eqref{ClassForm} and the embedding $W^{3,2}(\Omega)\hookrightarrow C^1(\overline{\Omega})$ we get
\begin{equation*}
	(\tilde u,\tilde F)\in L^\infty(0,T;C^1_0(\Omega)^d)\times L^\infty(0,T;C^1(\overline{\Omega})^{d\times d}).
\end{equation*}
Hence by the standard approximation arguments we find sequences
\begin{equation*}
	\{\tilde{u}^\delta\}\subset C^\infty([0,T];C^\infty_c(\Omega)^d),
	\{\tilde F^\delta\}\subset C^\infty([0,T];C^\infty(\overline{\Omega})^{d\times d})
\end{equation*}
with $\dvr \tilde u^\delta$, $\dvr\tilde F^\delta=0$ in $Q_T$, such that 
\begin{equation}\label{DeltaConv}
\begin{alignedat}{2}
	\tilde{u}^\delta&\to \tilde u&&\text{ in }L^2(0,T;C^1(\overline{\Omega})^d),\\
	\tder\tilde{u}^\delta&\to \tder\tilde u&&\text{ in }L^2((0,T)\times\Omega)^d,\\
	\tilde F^\delta&\to \tilde F &&\text{ in }L^2(0,T;C^1(\overline{\Omega})^{d\times d}),\\
	\tder\tilde F^\delta&\to \tder\tilde F &&\text{ in }L^2((0,T)\times\Omega)^{d\times d}
	\end{alignedat}
\end{equation}
as $\delta\to 0_+$. Setting $\psi=\tilde{u}^\delta$ and $\Psi=\tilde F^\delta$ in \eqref{WeakForms}, which is allowed due to properties of the latter functions, and applying convergences \eqref{DeltaConv} we obtain
\begin{align*}
	\left[(u,\tilde u)\right]_{\tau=0}^{\tau=t}&=\int_0^t (u,\tder \tilde u)+(u\otimes u-\nabla u-FF^\top,\nabla \tilde u) +\int_0^t\left\langle\mathcal{R}^M,\nabla\tilde u\right\rangle,\\
	\left[(F,\tilde F)\right]_{\tau=0}^{\tau=t}&=\int_0^t(F,\tder\tilde F)+(u\otimes F, \nabla\tilde F) +(\nabla uF,\tilde F).
\end{align*}
Using the fact that $(\tilde u,\tilde F)$ is the strong solution to \eqref{ClassForm} it follows that after obvious manipulations we arrive at
\begin{align*}
	\left[(u,\tilde u)\right]_{\tau=0}^{\tau=t}&=\int_0^t\bigl(u,-\dvr (\tilde u\otimes\tilde u)+\Delta \tilde u+\dvr(\tilde F\tilde F^\top)\bigr) +(u\otimes u-\nabla u-FF^\top,\nabla \tilde u) +\int_0^t\left\langle\mathcal{R}^M,\nabla\tilde u\right\rangle\\
	&=\int_0^t \Bigl(\bigl((u-\tilde u)\cdot\nabla\bigr) \tilde u,u-\tilde u\Bigr)-2 (\nabla u,\nabla\tilde u)-(\nabla u,\tilde F\tilde F^\top)-(\nabla \tilde u,FF^\top)+\int_0^t\left\langle\mathcal{R}^M,\nabla\tilde u\right\rangle
	\end{align*}
and similarly
\begin{align*}
\left[(F,\tilde F)\right]_{\tau=0}^{\tau=t}&=\int_0^t(F,-\dvr(\tilde u\otimes\tilde F)+\nabla\tilde u\tilde F)+(u\otimes F, \nabla\tilde F) +(\nabla uF,\tilde F)\\
	&=\int_0^t\Bigl(\bigl((u-\tilde u)\cdot\nabla\bigr)\tilde F,F-\tilde F\Bigr)+(\nabla\tilde u,\tilde FF^\top)+(\nabla u, F\tilde F^\top).
\end{align*}
Hence we obtain
\begin{align*}
	\left[\mathcal{E}\left(u,F|\tilde u,\tilde F\right)\right]_{\tau=0}^{\tau=t}=&[E(u,F)]_{\tau=0}^{\tau=t} +[E(\tilde u,\tilde F)]_{\tau=0}^{\tau=t}\\
	& -\int_0^t\Bigl(\bigl(u-\tilde u)\cdot\nabla\bigr) \tilde u,u-\tilde u\Bigr)+\Bigl(\bigl((u-\tilde u)\cdot\nabla\bigr)\tilde F,F-\tilde F\Bigr)\\
	&+\int_0^t(\nabla\tilde u,(F-\tilde F)(F-\tilde F)^\top)+(\nabla (\tilde u-u),(F-\tilde F)\tilde F^\top)\\
	&-\int_0^t\left\langle\mathcal{R}^M,\nabla\tilde u\right\rangle.
\end{align*}
Employing energy inequality \eqref{DSEI} for the dissipative solution $(u,F)$ and the energy equality for the strong solution $(\tilde u,\tilde F)$ we conclude 
\begin{align*}
	\left[\mathcal{E}\left(u,F|\tilde u,\tilde F\right)\right]_{\tau=0}^{\tau=t}+\mathcal{D}(t)
	\leq &c\int_0^t\bigl(\|\nabla \tilde u(s)\|_{C(\overline\Omega)}+\|\nabla \tilde F(s)\|_{C(\overline\Omega)}\\&+\|\tilde F(s)\|^2_{C(\overline\Omega)}\bigr)\mathcal{E}\left(u,F|\tilde u,\tilde F\right)(s)+\mathcal{D}(s)\ds.
\end{align*}
Taking into account that $(u,F)$ and $(\tilde u,\tilde F)$ emanate from the same initial data, i.e., $\mathcal{E}\left(u,F|\tilde u,\tilde F\right)\!(0)\allowbreak=0$, we deduce $u=\tilde u$ and $F=\tilde F$ a.e. in $(0,T)\times\Omega$ by the Gronwall lemma.

\section{Appendix}
The section is devoted to the construction of an orthonormal basis of $\CrlH(\Omega)$ via the spectral decomposition of a certain symmetric and compact operator.
 
Our intention is first to construct a certain operator on which we apply the following theorem summarizing results from \cite[Section D.5.]{Evans98}.
\begin{Theorem}\label{Thm:SACOp}
	Let $H$ be an infinite dimensional Hilbert space with the scalar product $(\cdot,\cdot)_H$, $K:H\to H$ be a linear compact operator, $\sigma(K)$ be the spectrum of $K$ and $\sigma_p(K)\subset\sigma(K)$ be the discrete spectrum of $K$. Then 
	\begin{enumerate}
		\item $0\in\sigma(K)$,
		\item $\sigma(K)\setminus\{0\}=\sigma_p(K)\setminus\{0\}$,
		\item either $\sigma(K)\setminus\{0\}$ is finite or $\sigma(K)\setminus\{0\}$ is a sequence tending to $0$.
		\item If $\lambda\in \sigma(K)\setminus\{0\}$ then the eigenspace $\ker(K-\lambda \IdOp)$ is finite dimensional.
		\item If $K$ is positive, i.e., $(Kv,v)_H\geq 0$ for all $v\in H$, then $\sigma(K)\subset [0,\infty)$.
		\item If $K$ is symmetric, i.e., $(Ku,v)_H=(u,Kv)_H$ for all $u,v\in H$, then $\sigma(K)\subset\eR$. Additionaly, if $H$ is separable then it possesses an orthonormal basis consisting of eigenvectors of $K$.
	\end{enumerate}
\end{Theorem}
We start with the analysis of the following Neumann eigenvalue problem in $W^{1,2}(\Omega)^{d\times d}$: Find a function $\Phi$ (and an associated $\pi$)  satisfying
\begin{alignat*}{2}
-\Delta \Phi+(\nabla \pi)^\top+\Phi=&\lambda\Phi&&\text{ in }\Omega,\\
\dvr \Phi=&0&&\text{ in }\Omega,\\
(\nabla\Phi -\pi\otimes \IdM)n=&0&&\text{ on }\partial\Omega
\end{alignat*}
understood in the sense: Find a $\Phi\in \WS(\Omega)$ such that
\begin{equation}\label{NPInt}
	(\nabla \Phi,\nabla \Xi)+(\Phi,\Xi)=\lambda(\Phi,\Xi)\text{ for all }\Xi\in \WS(\Omega).
\end{equation}
The results concerning the latter problem are summarized in the ensuing theorem.
\begin{Theorem}\label{Thm:Basis}
	Let $\Omega\subset\eR^d$ be a bounded Lipschitz domain. Then there is a sequence of eigenvalues $\{\lambda^j\}_{j=1}^\infty$ such that 
	\begin{equation*}
		1=\lambda^1=\ldots=\lambda^{d^2}\leq \lambda^{d^2+1}\leq\ldots\leq\lambda^j\leq \lambda^{j+1}, \lambda^j\to\infty\text{ as }j\to\infty
	\end{equation*}
	and a sequence of corresponding eigenfunctions $\{\Phi^j\}_{j=1}^\infty$ satisfying \eqref{NPInt}. Moreover, $\{\Phi^j\}_{j=1}^\infty$ forms an orthonormal basis of $\CrlH(\Omega)$ as well as an orthogonal basis of $\WS(\Omega)$.
\end{Theorem}
\begin{proof}
We consider an operator $B:\WS(\Omega)\to(\WS(\Omega))^*$ defined as
\begin{equation*}
\left\langle B\Phi,\Xi\right\rangle = (\nabla \Phi,\nabla \Xi)+(\Phi,\Xi). 
\end{equation*}
Obviously, $B$ is bounded and linear. As an immediate consequence of the Lax-Milgram theorem $B$ is an isomorphism of $\WS(\Omega)$ onto its dual. Let us denote by $S$ the restriction of $B^{-1}$ on $\CrlH(\Omega)$. As $\CrlH(\Omega)\hookrightarrow(\WS(\Omega))^*$ and $\WS(\Omega)\stackrel{C}{\hookrightarrow}\CrlH(\Omega)$ we deduce that $S$ is compact on $\CrlH(\Omega)$. Next we get for arbitrary $\Phi^1,\Phi^2\in \CrlH(\Omega)$
\begin{equation*}
	(S\Phi^1,\Phi^2)=(S\Phi^1,BS\Phi^2)=(\nabla S\Phi^1,\nabla S\Phi^2)+(S\Phi^1,S\Phi^2)=(BS\Phi^1,S\Phi^2)=(\Phi^1,S\Phi^2),
\end{equation*}
i.e. $S$ is symmetric. Since $S$ is linear and bounded on $\CrlH(\Omega)$, it is also self-adjoint. It is shown similarly that $(S\Phi,\Phi)\geq 0$ for any $\Phi\in\CrlH(\Omega)$. Therefore according to Theorem \ref{Thm:SACOp} there exists an orthonormal basis $\{\Phi^j\}_{j=1}^\infty$ of the separable Hilbert space $\CrlH(\Omega)$ consisting of eigenfunctions of the linear bounded compact symmetric and self-adjoint operator $S$ with corresponding eigenvalues $\{\mu^j\}_{j=1}^\infty\subset [0,\infty)$ such that $\mu^j\to 0$ as $j\to\infty$. Obviously, as $S\Phi=0$ implies $\Phi=0$, $0$ is not an eigenvalue of $S$. As an immediate consequence we observe that setting $\lambda^j=\frac{1}{\mu^j}$ for each $j\in\eN$ we get that $\{\lambda^j\}_{j=1}^\infty$ is a sequence of eigenvalues of $B$ with corresponding eigenfunctions $\{\Phi^j\}_{j=1}^\infty$. The fact that $\lambda^1=\ldots=\lambda^{d^2}=1$ follows from the observation that \eqref{NPInt} is satisfied with $\lambda=1$ and $\Phi$ being equal to each basis element of $\eR^{d\times d}$. Identity \eqref{NPInt} with $\lambda=\lambda^j$ and $\Phi=\Phi^j$ and $\Xi=\Phi^i$ also implies that $\{\Phi^j\}_{j=1}^\infty$ is an orthogonal sequence in $\WS(\Omega)$ since the left hand side of \eqref{NPInt} is in fact the scalar product on $\WS(\Omega)$. Moreover, it follows that $\overline{\spn\{\Phi^j\}_{j=1}^\infty}$ has only the trivial orthogonal complement in $\WS(\Omega)$. Hence $\{\Phi^j\}_{j=1}^\infty$ is also an orthogonal basis in $\WS(\Omega)$.
\end{proof}
\vspace{1em}
\noindent\textbf{Acknowledgement}\\
The research leading to these results was supported by DFG grant SCHL 1706/4-1.

\noindent This is a pre-print of an article published in Journal of Mathematical Fluid Mechanics. The final authenticated version is available online at: \href{https://doi.org/10.1007/s00021-019-0459-9}{https://doi.org/10.1007/s00021-019-0459-9}.

\providecommand{\bysame}{\leavevmode\hbox to3em{\hrulefill}\thinspace}
\providecommand{\MR}{\relax\ifhmode\unskip\space\fi MR }
\providecommand{\MRhref}[2]{%
  \href{http://www.ams.org/mathscinet-getitem?mr=#1}{#2}
}
\providecommand{\href}[2]{#2}

\end{document}